\documentclass[11pt,a4paper]{article}

\usepackage{eepic}

\usepackage{amsmath}
\usepackage{amssymb}
\usepackage{amsthm}
\usepackage{latexsym}
\usepackage{pstricks}
\usepackage{amsbsy}
\usepackage{amscd}
\usepackage{mathrsfs}
\usepackage{tipa}
\usepackage{txfonts}
\usepackage{amsfonts}
\usepackage{graphicx}
\usepackage{tabularx}
\usepackage{listings}
\usepackage{tikz}
\usepackage{hyperref}
\tikzstyle{dot} = [inner sep=0pt,thick,fill=black,circle,minimum size=2.5pt]
\tikzstyle{line} = [draw, -latex]
\newtheorem{thm}{Theorem}[section]
\newtheorem{cor}[thm]{Corollary}
\newtheorem{lem}[thm]{Lemma}

\newtheorem{exm}{Example}

\newtheorem{defn}[thm]{Definition}
\newtheorem{rem}[thm]{Remark}
\newtheorem{defn-prop}[thm]{Definition-Proposition}

\newcommand{\qihao}{\fontsize{7.25pt}{\baselineskip}\selectfont}

\usepackage[all,poly]{xy}
\usepackage[all]{xy}
\usepackage{xypic}

\begin{document}

\begin{center}
{\Large \bf Cluster automorphism groups of cluster algebras of finite type
\footnote{Supported by the NSF of China (Grant 11131001)}}

\bigskip

{\large Wen Chang and
 Bin Zhu}

\bigskip

{Dedicated to the memory of Andrei Zelevinsky}

\bigskip


\end{center}

\begin{abstract}
We study the cluster automorphism group $Aut(\mathcal{A})$ of a coefficient free cluster algebra $\mathcal{A}$ of finite type. A cluster automorphism of $\mathcal{A}$ is a permutation of the cluster variable set $\mathscr{X}$ that is compatible with cluster mutations. We show that, on the one hand, by the well-known correspondence between $\mathscr{X}$ and the almost positive root system $\Phi_{\geq -1}$ of the corresponding Dynkin type, the piecewise-linear transformations $\tau_+$ and $\tau_-$ on $\Phi_{\geq -1}$ induce cluster automorphisms $f_+$ and $f_-$ of $\mathcal{A}$ respectively; on the other hand, excepting type $D_{2n} (n\geqslant 2)$, all the cluster automorphisms of $\mathcal{A}$ are compositions of $f_+$ and $f_-$. For a cluster algebra of type $D_{2n} (n\geqslant 2)$, there exists an exceptional cluster automorphism induced by a permutation of negative simple roots in $\Phi_{\geq -1}$, which is not a composition of $\tau_+$ and $\tau_-$. By using these results and folding a simply laced cluster algebra, we compute the cluster automorphism group for a non-simply laced finite type cluster algebra. As an application, we show that $Aut(\mathcal{A})$ is isomorphic to the cluster automorphism group of the $FZ$-universal cluster algebra of $\mathcal{A}$.
\end{abstract}

\def\s{\stackrel}
\def\Longrightarrow{{\longrightarrow}}

\def\ggz{\Gamma}
\def\bz{\beta}
\def\az{\alpha}
\def\gz{\gamma}
\def\da{\delta}
\def\zz{\zeta}
\def\thz{\theta}
\def\ra{\rightarrow}

\def\FS{\mathfrak{S}}
\def\F1{\mathfrak{F}}
\def\A{\mathcal{A}}
\def\B{\mathcal{B}}
\def\C{\mathcal{C}}
\def\D{\mathcal{D}}
\def\F{\mathcal{F}}
\def\H{\mathcal{H}}
\def\I{\mathcal {I}}
\def\P{\mathbb{P}}
\def\T{\mathbb{T}}
\def\R{\mathcal{R}}
\def\Si{\Sigma}
\def\S{\Sigma}
\def\L{\mathscr{L}}
\def\l{\mathcal{l}}
\def\U{\mathscr{U}}
\def\V{\mathscr{V}}
\def\W{\mathscr{W}}
\def\X{\mathscr{X}}
\def\Y{\mathscr{Y}}
\def\M{{\bf{Mut}}}
\def\MM#1{{\bf{(CM#1)}}}
\def\t{{\tau }}
\def\b{\textbf{d}}
\def\K{{\cal K}}
\def\righta{\rightarrow}
\def\G{{\Gamma}}
\def\x{\mathbf{x}}
\def\ex{{\mathbf{ex}}}
\def\fx{{\mathbf{fx}}}
\def\Aut{\mbox{Aut}}
\def\add{\mbox{add}}
\def\coker{\mbox{coker}}
\def\End{\mbox{End}}
\def\Ext{\mbox{Ext}}
\def\Gr{\mbox{Gr}}
\def\Hom{\mbox{Hom}}
\def\id{\mbox{id}}
\def\ind{\mbox{ind}}
\def\Int{\mbox{Int}}
\def\deg{\mbox{deg}}
\def \text{\mbox}
\def\m{\multiput}
\def\mul{\multiput}
\def\c{\circ}
\def \text{\mbox}
\def\t{\tilde}

\newcommand{\homeo}{\textup{Homeo}^+(S,M)}
\newcommand{\homeoo}{\textup{Homeo}_0(S,M)}
\newcommand{\mg}{\mathcal{MG}(S,M)}
\newcommand{\mmg}{\mathcal{MG}_{\bowtie}(S,M)}
\newcommand{\Z}{\mathbb{Z}}
\newcommand{\Q}{\mathbb{Q}}
\newcommand{\N}{\mathbb{N}}

\hyphenation{ap-pro-xi-ma-tion}

\medskip


\tableofcontents

\section{Introduction}
Cluster algebras are introduced by Sergey Fomin and Andrei Zelevinsky in \cite{FZ02}; it has been showed that these algebras are linked to various areas of mathematics, see, for examples, \cite{GLS08, F10, L10, Re10, K12, M14}, and so on. However, as an algebra itself with combinatorial structure, it is natural and interesting to study the symmetries of a cluster algebra. For this, Assem, Schiffler and Shramchenko \cite{ASS12} introduced cluster automorphisms and the cluster automorphism group of a cluster algebra without coefficients. A cluster automorphism is an algebra automorphism which maps clusters to clusters, and commutes with the mutations. These concepts and some similar ones are studied in many papers \cite{S10,ASS12,BQ12,ASS13,BD13,KP13,N13, Z06,CZ15,CZ15b}. \\

It is well known that the classification of indecomposable cluster algebras of finite type corresponds to the Cartan-Killing classification of simple Lie algebras, equivalently, corresponds to the classification of root systems in Euclidean spaces\cite{FZ03}. More precisely, the set of cluster variables $\X$ of $\A$ is in bijection with the almost positive root set $\Phi_{\geq -1}$ of the corresponding root system. Note that a cluster automorphism of $\mathcal{A}$ is a permutation of the cluster variable set $\mathscr{X}$, which commutes with cluster mutations, so it is natural to ask what is the counter-part on the almost positive root system $\Phi_{\geq -1}$. Our first aim is to answer this question. For this we consider piecewise-linear transformations $\tau_+$ and $\tau_-$ on $\Phi_{\geq -1}$, which is introduced in \cite{FZ03a} to prove Zamolodchikov's periodicity conjecture that concerns Y-system\cite{Z91}. A Y-system is a class of recurrent functions defined by a Cartan matrix. These two transformations are kinds of `linearization' of recurrence relations in the Y-system. It is proved in \cite{FZ03a} that the group $D_\tau$ generated by $\tau_{\pm}$ is a dihedral group, and the finiteness of $D_\tau$ yields the periodicity of the Y-system.\\

Under the correspondence between $\mathscr{X}$ and $\Phi_{\geq -1}$, the piecewise-linear transformation $\tau_\pm$ on $\Phi_{\geq -1}$ induces a
permutation $f_\pm$ of $\mathscr{X}$. In subsection \ref{subsec:tau transform groups}, we show that, both $f_+$ and $f_-$ give cluster automorphisms of $\mathcal{A}$. Conversely, except type $D_{2n} (n\geqslant 2)$, all the cluster automorphisms of $\mathcal{A}$ are compositions of $f_+$ and $f_-$ (see Corollary \ref{cor:simply laced auto group and tau group} and Theorem \ref{thm:tau gp eq to auto gp for nonsimply laced}), and thus $Aut(\A)\cong D_{\tau}$. For a cluster algebra of type $D_{2n} (n\geqslant 2)$, there exists an exceptional cluster automorphism induced by a permutation of negative simple roots in $\Phi_{\geq -1}$, which is not generated by $\tau_+$ and $\tau_-$, and thus $Aut(\A)\cong D_{\tau}\times \Z_2$.\\

The cluster automorphism group of a simply laced cluster algebra of finite type is computed in \cite{ASS12}, by computing the automorphism group of the $AR$-quiver of the corresponding cluster category. We study the cluster automorphism group in subsection \ref{Cluster automorphism groups: non-simply laced cases} for a non-simply laced one by folding technique. The folding of a root system is a usual method in the studies of Lie algebras and quantum groups, it transforms a simply laced root system to a non-simply laced one. The folding technique is also used to study cluster algebras \cite{FZ03a,D08,Y09}. By using the results obtained in subsection \ref{subsec:tau transform groups} and folding a simply laced cluster algebra \cite{D08} (see arXiv:math/0512043v5 for an improved version of \cite{D08}), we compute in Theorem \ref{thm:tau gp eq to auto gp for nonsimply laced} the cluster automorphism group of a non-simply laced cluster algebra of finite type. \\

For a coefficient free cluster algebra $\A$, its universal cluster algebra $\A^{univ}$ is introduced in \cite{FZ07}, it is a universal object in the set of cluster algebras with principal part $\A$, in the view point of coefficient specialization. It follows from \cite{CZ15} that $\A^{univ}$ is gluing free, that is, any two coefficient rows in each exchange matrix of $\A^{univ}$ are not the same. Thus $Aut(\A^{univ})\subseteq Aut(\A)$. For a finite type cluster algebra, Fomin and Zelevinsky constructed a universal cluster algebra, which is a geometric cluster algebra with coefficients indexed by dual roots in ${\Phi^{\vee}}_{\geq -1}$ \cite{FZ07}. We call it the FZ-universal cluster algebra, and prove that $Aut(\A)\cong Aut(\A^{univ})$ in section \ref{Sec universal cluster algebra}.\\

The paper is organized as follows: we recall preliminaries on cluster algebras, cluster automorphisms and piecewise-linear transformations in section \ref{Preleminaries}. In section \ref{Sec root system}, we consider the relations between cluster automorphisms and piecewise-linear transformations, and compute the cluster automorphism groups of non-simply laced cluster algebras of finite type. We prove in section \ref{Sec universal cluster algebra} the isomorphism between the cluster automorphism group of a cluster algebra of finite type and the cluster automorphism group of its universal cluster algebra.

\section{Preliminaries}
\label{Preleminaries}

\subsection{Cluster algebras}
We recall basic definitions and properties on cluster algebras in this subsection.
\begin{defn}\cite{FZ02}(Labeled seeds).\label{def: labeled seeds}
A \emph{labeled seed} is a triple $\Sigma=(\ex,\fx,B)$, where
\begin{itemize}
\item  $\ex=\{x_{1},x_{2},\cdot \cdot \cdot ,x_{n}\}$ is a set with n elements;
\item  $\fx=\{x_{n+1},x_{n+2},\cdot \cdot \cdot ,x_{m}\}$ is a set with $m-n$ elements;
\item  $B=(b_{x_jx_i})_{m\times n}\in M_{m\times n}(\Z)$ is a matrix labeled by $(\ex\sqcup\fx)\times \ex$, and it is extended skew-symmetrizable, that is, there exists a diagonal matrix $D$ with positive integer entries such that $DB^{ex}$ is skew-symmetric, where ${B}^{ex}$ is a submatrix of $B$ consisting of the first n rows.
\end{itemize}
\end{defn}
The set $\x=\ex \sqcup \fx$ is the \emph{cluster} of $\S$, B is the \emph{exchange matrix} of $\S$. We also write $b_{ji}$ to an element $b_{x_jx_i}$ in $B$ for brevity.
The elements in $\x$ ($\ex$ and $\fx$ respectively) are the \emph{cluster variables} (the \emph{exchangeable variables} and the \emph{frozen variables} respectively) of $\S$. We also write $(\x,B)$ to a labeled seed $(\ex,\emptyset,B)$. The labeled seed $\S^{ex}=(\ex,{B}^{ex})$ is called the \emph{principal part} of $\S$. The rows of $B^{ex}$ are called \emph{exchangeable rows} of $B$, and the rest ones are called \emph{frozen rows} of $B$. We always assume through the paper that both $B$ and $B^{ex}$ are indecomposable matrices, and we also assume that $n>1$ for convenience. Given an exchangeable cluster variable $x_k$, we produce a new labeled seed by a mutation.

\begin{defn}\cite{ADS13,FZ02}(Seed mutations).\label{def: mutation}
The labeled seed $\mu_k(\S)=(\mu_k(\ex),\mu_k(\fx),\mu_k(B))$ obtained by the \emph{mutation} of $\S$ in the direction $k$ is given by:
\begin{itemize}
				\item $\mu_k(\ex) = (\ex \setminus \{x_k\}) \sqcup \{x'_k\}$ where
				$$x_kx'_k = \prod_{\substack{1\leqslant j\leqslant m~; \\ b_{jk}>0}} {x_j}^{b_{jk}} + \prod_{\substack{1\leqslant j\leqslant m~; \\ b_{jk}<0}} {x_j}^{-b_{jk}}.$$
				\item $\mu_k(\fx)=\fx$.
				\item $\mu_k(B)=(b'_{ji})_{m\times n} \in M_{m\times n}(\Z)$ is given by
					$$b'_{ji} = \left\{\begin{array}{ll}
						- b_{ji} & \textrm{ if } i=k \textrm{ or } j=k~; \\
						b_{ji} + \frac 12 (|b_{ji}|b_{ik} + b_{ji}|b_{ik}|) & \textrm{ otherwise.}
					\end{array}\right.$$	
\end{itemize}
It is easy to check that the mutation is an involution, that is $\mu_k\mu_k(\S)=\S$.
\end{defn}

\begin{defn}\cite{FZ07}(n-regular patterns).\label{def:n-regular patterns}
An n-regular tree $\T_n$ is diagram, whose edges are labeled by $1,2,\cdots,n$, such that the $n$ edges emanating from each vertex receive different labels.
A \emph{cluster pattern} is an assignment
of a labeled seed $\Sigma_t=(\ex_t, \fx_t, B_t)$ to every vertex $t \in \T_n$, so that the seeds assigned to the
endpoints of any edge labeled by $k$ are obtained from each
other by the seed mutation in direction~$k$.
The elements of $\Sigma_t$ are written as follows:
\begin{equation}
\label{eq:seed-labeling}
\ex_t = (x_{1;t}\,,\dots,x_{n;t})\,,\quad
\fx_t = (x_{n+1;t}\,,\dots,x_{m;t})\,,\quad
B_t = (b^t_{ij})\,.
\end{equation}
Clearly $\fx_t=\fx$ for any $t\in \T_n$. Denote $\x_t=\ex_t\sqcup\fx_t$. Note that $\T$ is in fact determined by any given labeled seed on it.
\end{defn}

Now we are ready to define cluster algebras.
\begin{defn}\cite{FZ07}(Cluster algebras).\label{def: cluster algebras}
Given a seed $\S$ and a cluster pattern $\T_n$ associated to it, we denote
\begin{equation}
\label{eq:cluster-variables}
\X
= \bigcup_{t \in \T_n} \x_t
= \{ x_{i,t}\,:\, t \in \T_n\,,\ 1\leq i\leq m \} \ ,
\end{equation}
the union of clusters of all the seeds in the pattern.
We call the elements $x_{i,t}\in \X$ the \emph{cluster variables}.
The \emph{cluster algebra} $\A$ associated with $\S$ is the $\Z$-subalgebra of the rational function field $\F=\Q(x_{1},x_{2},\cdot \cdot \cdot ,x_{m})$,
generated by all cluster variables: $\A = \Z[\X]$.
We call the elements in $\fx$ the coefficients of $\A$. We call the cluster algebra $\A^{ex}$ defined by $\S^{ex}$ the principal part of $\A$. Note that $\A^{ex}$ is coefficient free.
\end{defn}

\begin{rem}
\begin{enumerate}
\item Since we consider geometric cluster algebras, the most of above concepts are slightly different from the ones in \cite{FZ02,FZ07}. In particular, all the coefficients are non-invertible in $\A$.
\item For the exchange matrix $B$ in a labeled seed $\Sigma=(\ex,\fx,B)$, one can associate it to an ice valued quiver $Q(B)$ ($Q$ for brevity), whose vertices are labeled by cluster variables in $\x$, with frozen vertices labeled by frozen variables, and arrows and values are assigned by $B$ (see Example \ref{exm:quiver and matrix}, we refer to \cite{K12,CZ14} for details). Then the principal part $B^{ex}$ corresponds the principal part $Q^{ex}$ of $Q$, where $Q^{ex}$ is a valued quiver. We define the mutation of $Q$ at a vertex corresponding to $x\in\ex$ by the mutation of $B$ at $x$, that is, $\mu_{x}(Q)=Q(\mu_{x}(B))$. We also write $(\ex,\fx,Q)$ to the labeled seed $(\ex,\fx,B)$, and write $\A_Q$ to the cluster algebra defined by $\Sigma$.
\end{enumerate}
\end{rem}

\begin{exm}\label{exm:quiver and matrix}
Let $B$ be the following matrix, whose principal part is a skew-symmetrizable matrix with diagonal matrix $D=diag\{2,2,1,1\}$. The quiver corresponding to $B$ is $Q$, where we frame the frozen vertices.
\begin{center}
{\begin{tikzpicture}
\node[] (C) at (-3,-3)  {$B$
$~=~$
                        $\left(
                           \begin{array}{c}
                             B^{ex} \\
                             B' \\
                           \end{array}
                         \right)$
$~=~$
                        $\left(
                          \begin{array}{ccccc}
                            0 & 1 & 0 & 0\\
                            -1 & 0 & -1 & 0\\
                            0 & 2 & 0 & 2\\
                            0 & 0 & -2 & 0\\
                            0 & 0 & 0 & -1\\
                          \end{array}
                        \right)$};
\end{tikzpicture}}
\end{center}

\begin{center}
{\begin{tikzpicture}
\node[] (C) at (-2.5,0)  {$Q~:$};
\node[] (C) at (-1.5,0)  {$1$};
\node[] (C) at (0,0)  {$2$};	
\node[] (C) at (1.5,0)  {$3$};
\node[] (C) at (3,0)  {$4$};
\node[] (C) at (4.5,0)  {$\frame{5}$};
\node[] (C) at (0.77,0.3)  {\qihao{(2,1)}};
\draw[<-,thick] (-0.2,0) -- (-1.3,0);
\draw[<-,thick] (0.2,0) -- (1.3,0);
\draw[->,thick] (1.7,0.05) -- (2.8,0.05);
\draw[->,thick] (1.7,-0.05) -- (2.8,-0.05);
\draw[->,thick] (3.2,0) -- (4.3,0);
\end{tikzpicture}}
\end{center}
\end{exm}

\begin{defn}\cite{CZ15}(Gluing free labeled seeds).\label{def: gluing free labeled seeds}
Let $\Sigma=(\ex,\fx,B)$ be a labeled seed. We say that two frozen cluster variables $x_j$ and $x_k$ are \emph{strictly glueable}, if $b_{ji}= b_{ki}$ for any exchangeable cluster variable $x_i$. The labeled seed $\Sigma$ is called \emph{gluing free}, if any two frozen cluster variables are not strictly glueable.
\end{defn}
Gluing free labeled seeds are introduced in \cite{CZ15} to study cluster automorphisms of cluster algebras with coefficients. Note that a labeled seed is gluing free if and only if any two frozen rows of $B$ are different. Then a mutation of a gluing free labeled seed is still gluing free \cite{CZ15}. Thus we have the following well-defined gluing free cluster algebra.
\begin{defn}\cite{CZ15}(Gluing free cluster algebras).\label{def: gluing free cluster algebras}
We say a cluster algebra is \emph{gluing free}, if its labeled seeds are gluing free.
\end{defn}
\begin{defn}\cite{FZ07}(Seeds).
\label{def:seeds}
Given two labeled seeds $\Sigma=(\ex, \fx, B)$
and $\Sigma'=(\ex', \fx', B')$, we say that they
define the same \emph{seed}
if $\Sigma'$ is obtained from $\Sigma$ by simultaneous relabeling of the
sets $\ex$ and $\fx$ and the corresponding relabeling of
the rows and columns of~$B$.
\end{defn}
We denote by $[\S]$ the seed represented by a labeled seed $\S$. The cluster $\x$ of a seed $[\S]$ is an unordered $m$-element set. For any $x\in\ex$, there is a well-defined mutation $\mu_{x}([\S])=[\mu_{k}(\S)]$ of $[\S]$ at direction $x$, where $x=x_k$. For two extended skew-symmetrizable matrices $B$ and $B'$ of the same rank, we say $B\cong B'$, if $B'$ is obtained from $B$ by simultaneous relabeling of the exchangeable rows and corresponding columns and the relabeling of the frozen rows. Then two exchange matrices are isomorphic if they are in labeled seeds representing the same seed, and an isomorphism of two exchange matrices induces an isomorphism of corresponding ice valued quivers. For convenience, in the rest of the paper, we denote by $\S$ the seed $[\S]$ represented by $\S$.
\begin{defn}\cite{FZ07}(Exchange graphs).
\label{def:exchange-graph}
The \emph{exchange graph} of a cluster algebra is the $n$-regular graph whose vertices are the seeds of the cluster algebra and whose
edges connect the seeds related by a single mutation. We denote by $E_\A$ the exchange graph of a cluster algebra $\A$.
\end{defn}
Clearly, the exchange graph of a cluster algebra is a quotient graph of the exchange pattern, its vertices are equivalent classes of labeled seeds.
Note that the edges in the exchange graph lost the `color' of labels.
The exchange graph is not necessary a finite graph, if it is finite, then we say the corresponding cluster algebra (and its exchange pattern) is of \emph{finite type}. The classification of cluster algebras of finite type is given in \cite{FZ02}, they correspond to finite root system. We will recall the correspondence in the next subsection.\\

The following specialization is firstly considered by Fomin and Zelevinsky, and it is viewed as a kind of morphism between cluster algebras in \cite{ADS13}. We will use this concept in Lemma \ref{lem:relations between two groups} and Theorem \ref{thm:auto gp of finit type univ alg}.
\begin{defn-prop}\cite{FZ07,ADS13}\label{def: specialization}
Let $\A$ be a cluster algebra with a seed $(\ex,\fx,B)$, and $\A^{ex}$ is the principal part of $\A$ with a seed $(\ex,B^{ex})$. We define $S$ by $S(x)=x$ for $x\in \ex$ and $S(x)=1$ for $x\in \fx$, then it induces an algebra homomorphism $S'$ from $\A$ to $\A^{ex}$, we call it a specialization.
\end{defn-prop}

\subsection{Piecewise-linear transformations}

We recall bipartite seeds and piecewise-linear transformations of finite root system from \cite{FZ03a,FZ07}.

\begin{defn}\cite{FZ07}(Bipartite seeds).\label{def: Bipartite seeds}
We call a (labeled) seed $\S=(\ex,\fx, B)$ \emph{bipartite}, if the principal part $B^{ex}$ of $B=(b_{ji})$ is bipartite, that is, there exists a map $\varepsilon :[1, n]\to \{1,-1\}$ such that, for all $1\leqslant i,j \leqslant n$,
\begin{equation}
\label{eq:bipartite}
b_{ji} > 0 \Rightarrow
\begin{cases}
\varepsilon(j) = 1 \,, \\
\varepsilon(i) = -1\,.
\end{cases}
\end{equation}
\end{defn}

For a square matrix $B$, its \emph{Cartan counterpart}(see \cite{FZ03}(1.6)) is $A=A(B)=(a_{ij})$, where
\begin{equation}
\label{eq cartan}
a_{ij} =
\begin{cases}
2 & \text{if $i=j$;} \\ 
- |b_{ij}| & \text{if $i\neq j$.}
\end{cases}
\end{equation}
It is proved in \cite{FZ03}(Theorem 1.4) that a cluster algebra is of finite type if and only if there exists a seed of the cluster algebra, such that the Cartan counterpart of the principal part of its exchange matrix is a finite type Cartan matrix. In this subsection we always assume that $\S=(\x,B)$ is a bipartite seed without frozen variables, such that $A(B)$ is a finite type Cartan matrix. Then the valued quiver $Q$ is a bipartite quiver, that is, any vertex of $Q$ is a source or a sink.\\


Note that if $b_{ij}=0$, then $\mu_i\mu_j=\mu_j\mu_i$, thus we have the following well-defined compositions of mutations on $\S$:
\begin{equation}
\label{eq mupm}
\mu_+ = \prod_{\varepsilon(k) = 1} \mu_k \,,\qquad
\mu_- = \prod_{\varepsilon(k) = -1} \mu_k \,.
\end{equation}

Clearly, $\mu_\pm$ is an involution, and $\mu_\pm (B)=-B$, thus $\mu_\pm(\S)$ is also a bipartite seed.

\begin{defn}\cite{FZ07}(Bipartite belt)
\label{def: bipartite belt}
For $r > 0$, we define
\begin{align}
\label{eq:big-circle-seeds+}
\S_r=(\x_{r}, (-1)^r B) &=
\underbrace{\mu_\pm \cdots \mu_- \mu_+ \mu_-}_{r \text{~factors}}
(\S);\\
\label{eq:big-circle-seeds-}
\Sigma_{-r}=(\x_{- r}, (-1)^r B) &=
\underbrace{\mu_\mp \cdots \mu_+ \mu_- \mu_+}_{r \text{~factors}}
(\S).\end{align}
We call the belt consisting of these seeds $\S_r=(\x_{r}, (-1)^r B)$ the \emph{bipartite belt}. Denote $\x_r= (x_{1;r}, \dots, x_{n;r})$ for each $\S_r$.
\end{defn}

For the Cartan matrix $A$, write $\Pi=\{\alpha_1,\alpha_2,\cdots,\alpha_n\}$
to the set of its positive simple roots, with root lattice $E=\Z\Pi$. Define $\Phi_{\geq -1}$ as the almost positive root system, it consists of the positive roots and the negative simple roots of $A$. For any $i \in [1, n]=\{1,2,\cdot\cdot\cdot,n\}$, $s_i$ is the simple reflection in the corresponding Weyl group $W$, that is, $s_i(\alpha_j)=\alpha_j-a_{ij}\alpha_i$ for each $j \in [1, n]$. The Coxeter number of $W$ is $h$, and the longest element is $w_0$.
We denote a root $\alpha$ by $\sum_{j=1}^{n}[\alpha:\alpha_j]\alpha_j$, where $[\alpha:\alpha_j]$ is the coefficient of $\alpha$ corresponding to the positive simple root $\alpha_j$. For each element $\alpha=\sum_{j=1}^{n}[\alpha:\alpha_j]\alpha_j$ in $\Phi_{\geq -1}$, we define a vector ${\bf{d}}(\alpha)=([\alpha:\alpha_1]_+,\cdots,[\alpha:\alpha_n]_+)\in E$, where
$[\alpha:\alpha_j]_+=max([\alpha:\alpha_j],0)$. Then by the classification of finite type cluster algebras \cite{FZ03},(Theorem 1.9),
\begin{equation}\label{equ:cluster variables and roots}
\alpha\mapsto x_\alpha=\frac{P_\alpha(\x)}{\x^{\bf{d}(\alpha)}}
\end{equation}
gives a one-to-one correspondence between almost positive roots in $\Phi_{\geq -1}$ to the cluster variables of $\A$, where $P_\alpha$ is a polynomial with non-negative integer coefficients.

\begin{defn}\cite{FZ03a}\label{piecewise linear transformation}
A piecewise linear transformation $\sigma_i: E\to E$ is defined by:
\begin{equation}
\sigma_i(\alpha) =
\begin{cases}
\alpha & \text{if $\alpha=-\alpha_j\neq -\alpha_i$;} \\ 
s_{i}(\alpha) & \text{others.}
\end{cases}
\end{equation}
\end{defn}
Then for a root $\alpha=\sum_{j\in\bf{I}}[\alpha:\alpha_j]\alpha_j$,
\begin{equation}
\label{sigma i}
[\sigma_i(\alpha):\alpha_{i'}] =
\begin{cases}
[\alpha:\alpha_{i'}] & \text{if $i'\neq i$;} \\ 
-[\alpha:\alpha_{i}]-\sum_{j\neq i}a_{ij}[\alpha:\alpha_j]_+ & \text{if $i'=i$.}
\end{cases}
\end{equation}
From the definition, $\sigma_i\sigma_j=\sigma_j\sigma_i$ when $\varepsilon(i)=\varepsilon(j)$, and we define
$$\tau_{\pm}=\prod_{\varepsilon(i)={\pm}1}\sigma_i.$$
It is easy to check that
\begin{equation}
\label{tau}
[\tau_\varepsilon(\alpha):\alpha_{i}] =
\begin{cases}
[\alpha:\alpha_{i}] & \text{if $\varepsilon(i)\neq\varepsilon$;} \\ 
-[\alpha:\alpha_{i}]-\sum_{j\neq i}a_{ij}[\alpha:\alpha_j]_+ & \text{if $\varepsilon(i)=\varepsilon$.}
\end{cases}
\end{equation}
The following lemma can be checked straightforwardly:
\begin{lem}\label{lem:tau group}{\em[Proposition 2.4\cite{FZ03a}]}
\begin{enumerate}
\item Both transformations $\tau_+$ and $\tau_-$ are involutions, and preserve $\Phi_{\geq -1}$;
\item The bijection $\alpha\mapsto \alpha^{\vee}$ between $\Phi_{\geq -1}$ and  ${\Phi^{\vee}}_{\geq -1}$ is $\tau_{\pm}$ equivariant.
\end{enumerate}
\end{lem}
\begin{exm}
Let $\Phi$ be a root system of type $A_2$. We assume that $\varepsilon(1)=1$ and $\varepsilon(2)=-1$. Its almost positive root system $\Phi_{\geq -1}$ is depicted as follows:
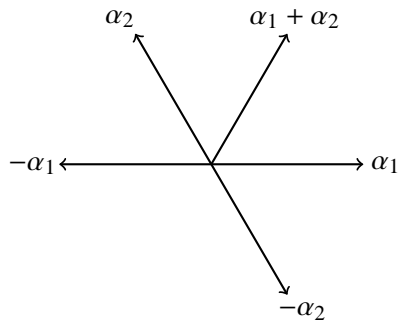
\begin{center}
\begin{tikzpicture}
                        \node[] () at (-2.35,0) {$-\alpha_1$};
                        \node[] () at (2.3,0) {$\alpha_1$};
					    \draw[<->,thick] (-2,0) -- (2,0);
					    \draw[<->,thick] (-1,1.73) -- (1,-1.73);
					    \draw[->,thick] (0,0) -- (1,1.73);
                        \node[] () at (-1.2,1.93) {$\alpha_2$};
                        \node[] () at (1.2,-1.93) {$-\alpha_2$};
                        \node[] () at (1.1,1.93) {$\alpha_1+\alpha_2$};
\end{tikzpicture}
\end{center}
Then $\tau_+=\sigma_1$ and $\tau_-=\sigma_2$ act on $\Phi_{\geq -1}$ in the following way:
\begin{center}
\begin{tikzpicture}
                        \node[] () at (0,0) {$\alpha_1+\alpha_2$};
                        \node[] () at (2.1,0) {$\alpha_2$};
                        \node[] () at (3.8,0) {$-\alpha_2$};
                        \node[] () at (-2.1,0) {$\alpha_1$};
                        \node[] () at (-3.8,0) {$-\alpha_1$};
					    \draw[<->,thick] (0.8,0) -- (1.8,0);
					    \draw[<->,thick] (-0.8,0) -- (-1.8,0);
					    \draw[<->,thick] (2.4,0) -- (3.4,0);
					    \draw[<->,thick] (-2.4,0) -- (-3.4,0);

                        \node[] () at (1.4,0.25) {$\sigma_1$};
                        \node[] () at (3,0.25) {$\sigma_2$};
                        \node[] () at (-1.2,0.25) {$\sigma_2$};
                        \node[] () at (-2.8,0.25) {$\sigma_1$};
                        \draw[->,thick] (-3.6,-0.2)..controls(-3.8,-0.6)..(-4,-0.2);
                        \draw[->,thick]                         (3.6,-0.2)..controls(3.8,-0.6)..(4,-0.2);
                        \node[] () at (3.9,-0.8) {$\sigma_1$};
                        \node[] () at (-3.7,-0.8) {$\sigma_2$};
\end{tikzpicture}
\end{center}
\end{exm}
We denote by $D_\tau$ the group generated by $\tau_+$ and $\tau_-$, it is a subgroup of the symmetric group of $\Phi_{\geq -1}$, and we call it the  $\tau$-\emph{transform group} of $\Phi_{\geq -1}$.

\begin{lem}{\em[Theorem 2.6\cite{FZ03a}]}\label{lem:structure of tau group}
\begin{enumerate}
\item Every $D$-orbit in $\Phi_{\geq -1}$ has a nonempty intersection with $-\Pi$. More specifically, the correspondence $\Omega\mapsto \Omega\cap (-\Pi)$ is a bijection between the $D$-orbits in $\Phi_{\geq -1}$ and the $<-\omega_0>$-orbits in $(-\Pi)$.
\item $D_\tau$ is a dihedral group of order $(h+2)$ or $2(h+2)$. The order of $\tau_-\tau_+$ is equal to $(h+2)/2$ if $w_0=-1$, and is equal to $(h+2)$ otherwise.
\end{enumerate}
\end{lem}

Then we have table \ref{table:tau group} on the $\tau$-transform groups.

\begin{table}[ht]
\begin{equation*}
\begin{array}{cccc}
\textrm{Dynkin type} & w_0=-1 & \textrm{Coxeter number} & \textrm{$\tau$-transform group $D_\tau$} \\
\hline
A_n & \textrm{No} & n+1 & D_{n+3} \\
 B_n & \textrm{Yes} & 2n &  D_{n+1}\\
 C_n & \textrm{Yes} & 2n & D_{n+1}\\
 D_n & \begin{cases}
\textrm{n odd} & \text{No} \\ 
\textrm{n even} & \text{Yes}
\end{cases}
& 2(n-1) & \begin{cases}
D_{2n} &\text{n odd} \\ 
D_{n} & \textrm{n even}
\end{cases} \\
E_6 & \textrm{No}  & 12 & D_{14}\\
 E_7 & \textrm{Yes} & 18 & D_{10}\\
E_8 & \textrm{Yes} & 30 & D_{16}\\
F_4  & \textrm{Yes} & 12 & D_{7}\\
G_2 & \textrm{Yes} &  6 & D_{4}\\
\end{array}
\end{equation*}
\smallskip
\caption{The $\tau$-transform group of finite type almost positive root system}
\label{table:tau group}
\end{table}

\begin{defn}{\em[Definition 10.2\cite{FZ07}]}\label{def:d-vectors}
For any $i\in [1,n]$, and $m$ such that $\varepsilon(i) = (-1)^m$, we define ${\bf{d}}(i;m)\in E$, by setting, for all $r\geqslant 0$:
\begin{align}
\label{eq:dems-m-positive-tau}
 {\bf{d}}(i;r) &=
\underbrace{\tau_- \tau_+ \cdots \tau_{\varepsilon(i)}}_{r
  \text{~factors}}(-\alpha_i) \qquad\, \text{for $\varepsilon(i) = (-1)^r$;}\\
\label{eq:dems-m-negative-tau}
 {\bf{d}}(j;-r-1) &=
\underbrace{\tau_+ \tau_- \cdots \tau_{\varepsilon(j)}}_{r
  \text{~factors}}(-\alpha_j) \qquad \text{for $\varepsilon(j) = (-1)^{r-1}$.}
\end{align}
\end{defn}
By \cite{FZ07} (Proposition 9.3),
\begin{equation}
\label{sqcup of d-vectors}
\Phi_{\geq -1}=\bigsqcup_{i\in[1,n]} \bigsqcup_{-h-1\leqslant r \leqslant h} \b(i;r)
\end{equation}
and
\begin{equation}
\label{eq:d-vector under longest element}
\begin{cases}
\b(i;-h-2)=-\alpha_{i^*} & \text{if $\varepsilon(i) = (-1)^h$;} \\ 
\b(j;h+1)=-\alpha_{j^*} & \text{if $\varepsilon(j) = (-1)^{h-1}$,}
\end{cases}
\end{equation}
where $i \mapsto i^*$ is the involution induced by the longest element $w_0 \!\in\! W$: $w_0(\alpha_i)\!=\! - \alpha_{i^*}$.
\begin{lem}{\em[Corollary 10.6 \cite{FZ07}]}\label{d-vectors on bipartite belt}
Each cluster variable $x_{i;r}$ on the bipartite belt \ref{def: bipartite belt} can be written as
\[
x_{i;r}=\frac{P_{i;r}(\x)}{\x^{{\bf{d}}(i;r)}}\,,
\]
where $P_{i;r}$ is a polynomial with non-zero constant term.
\end{lem}

\subsection{Automorphism groups}
In this subsection, we recall the cluster automorphism group of a cluster algebra, and the automorphism group of the corresponding exchange graph. Firstly we define cluster automorphisms, which are introduced in \cite{ASS12} for cluster algebras without coefficients, and in \cite{CZ15} for cluster algebras with coefficients.
\begin{defn}\label{cluster automorphisms}\cite{ASS12,CZ15}(\emph{Cluster automorphisms})
For a cluster algebra $\A$ and a $\Z$-algebra automorphism $f:\A\to\A$, we call $f$ a \emph{cluster automorphism}, if there exists a labeled seed $(\ex,\fx,B)$ of $\A$ such that the following conditions are satisfied:
\begin{enumerate}
\item $f(\x)=f(\ex)\sqcup f(\fx)$ is a cluster, where $f(\ex)$ is the exchangeable part and $f(\fx)$ is the frozen part;
\item $f$ is compatible with mutations, that is, for every $x\in \ex$ and $y\in \x$, we have
$$f(\mu_{x,\x}(y))=\mu_{f(x),f(\x)}(f(y)).$$
\end{enumerate}
\end{defn}
Then a cluster automorphism maps a labeled seed $\S=(\ex,\fx,B)$ to a labeled seed $\S'=(\ex',\fx',B')$. Note that in a labeled seed, the cluster is an ordered set, then the second item in above definition yields that $B'= B$ or $B'=-B$. In fact, under our assumption that both $B$ and $B^{ex}$ are indecomposable, we have the following
\begin{lem}\cite{ASS12,CZ15}\label{lem:equivalent discription of clus-auto}
A $\Z$-algebra automorphism $f:\A\to\A$ is a cluster automorphism if and only if one of the following conditions is satisfied:
\begin{enumerate}
\item there exists a labeled seed $\S=(\ex,\fx,B)$ of $\A$, such that $f(\x)$ is the cluster in a labeled seed $\S'=(\ex',\fx',B')$ of $\A$ with $B'= B$ or $B'=-B$;
\item for every labeled seed $\S=(\ex,\fx,B)$ of $\A$, $f(\x)$ is the cluster in a labeled seed $\S'=(\ex',\fx',B')$ with $B'= B$ or $B'=-B$.
\end{enumerate}
\end{lem}
We call this cluster automorphism such that $B= B'$ ($B= -B'$ respectively) the \emph{direct cluster automorphism} (\emph{inverse cluster automorphism} respectively). Clearly, all the cluster automorphisms of a cluster algebra $\A$ compose a group with homomorphism compositions as multiplications.
We call this group the \emph{cluster automorphism group} of $\A$, and denote it by $Aut(\A)$. We call the group $Aut^{+}(\A)$ consisting of the direct cluster automorphisms of $\A$ the \emph{direct cluster automorphism group} of $\A$, which is a subgroup of $Aut(\A)$ of index at most two\cite{ASS12,CZ15}.

\begin{defn}(Automorphism of exchange graphs)\cite{CZ15}
An automorphism of the exchange graph $E_\A$ of a cluster algebra $\A$ is an automorphism of $E_\A$ as a graph, that is, a permutation $\sigma$ of the vertex set, such that the pair of
vertices $(u,v)$ forms an edge if and only if the pair $(\sigma(u),\sigma(v))$ also forms an edge.
\end{defn}

It is clear that the natural composition of two automorphisms of $E_\A$ is again an automorphism of $E_\A$. We define an \emph{automorphism group
$Aut(E_\A)$} of $E_\A$ as a group consisting of automorphisms of $E_\A$ with compositions of automorphisms as multiplications.\\

\begin{lem}\label{lem:relations between two groups}
Let $\A$ be a finite type cluster algebra with a seed $(\ex,\fx,B)$, and $\A^{ex}$ is the principal part of $\A$ with a seed $(\ex,B^{ex})$. Let $S'$ be the specialization from $\A$ to $\A^{ex}$. Assume that $\A$ is gluing free, then
\begin{enumerate}
\item for a cluster automorphism $f$ of $\A$, the map $S'\circ f|_{\ex}: x\to S'(f(x))$ induces a cluster automorphism of $\A^{ex}$, and thus $Aut(\A)\subseteq Aut(\A^{ex})$;
\item a cluster automorphism of $\A$ maps clusters to clusters, and induces an automorphism of exchange graph $E_\A$, moreover, $Aut(\A)\subseteq Aut(E_\A)$.
\end{enumerate}
\end{lem}
\begin{proof}
These two statements are proved for cluster algebra defined by extended skew-symmetric matrix in Theorem 3.16 and Theorem 3.14 in \cite{CZ15} respectively. The key points of those proofs are as follows:
\begin{enumerate}
\item Clusters of $\A$ determine the seeds. Therefore, the vertices of the exchange graph are (unlabeled) clusters of $\A$.
\item The exchange graph $E_\A$ is independent on the choice of coefficients, that is, for another cluster algebra $\A'$ with the same principal part $\A^{ex}$ as $\A$ has, there is a canonical isomorphism $E_\A\cong E_\A'$.
\item $\A$ is gluing free. Thus there are no symmetries of $\A$ induced by permutations of the frozen cluster variables of $\A$.
\end{enumerate}
Now, note that these first two properties are also true for skew-symmetrizable cluster algebras of finite type (see in \cite{FZ03}). Thus the proofs are similar to the proofs in \cite{CZ15}.
\end{proof}

\begin{rem}
Generally, even for a coefficient free cluster algebra $\A$, the group $Aut(\A)$ may be a proper subgroup of $Aut(E_{\A})$, for example, cluster algebras of types $B_2$, $C_2$, $G_2$ and $F_4$ (see Example 2 \cite{CZ15b}). However, it is proved that these two groups are isomorphic with each other, if $\mathcal{A}$ is of finite type, excepting types of rank two and type $F_4$ (Theorem 3.7 in \cite{CZ15b}), or $\mathcal{A}$ is of skew-symmetric finite mutation type (Theorem 3.8 in \cite{CZ15b}).
\end{rem}

In the end of this section, we list the cluster automorphism groups of simply laced cluster algebras of finite type in table \ref{table:simply laced auto group}. These groups are computed in \cite{ASS12} by using the cluster category.
\begin{table}[ht]
\begin{equation*}
\begin{array}{cc}
\textrm{Dynkin type} & \textrm{Cluster automorphism group $Aut(\A)$} \\
\hline
A_{n} & D_{n+3}\\
D_4 & D_4\times S_3 \\
D_n (n\geqslant 5) & D_n\times\Z_2\\
E_6 & D_{14}\\
E_7 & D_{10}\\
E_8 & D_{16}\\
\end{array}
\end{equation*}
\smallskip
\caption{Cluster automorphism groups of simply laced cluster algebras of finite type}
\label{table:simply laced auto group}
\end{table}

\section{Cluster automorphisms and piecewise-linear transformations}
\label{Sec root system}
Firstly, in this section, we consider relations between cluster automorphism groups of coefficients free cluster algebras of finite type and the corresponding $\tau$-transform groups. Then we compute cluster automorphism groups for non-simply laced coefficients free finite type cluster algebras.
\subsection{$\tau$-transform groups}\label{subsec:tau transform groups}

In this subsection we fix a coefficient free cluster algebra $\A$ of finite type.
As mentioned in Lemma \ref{lem:relations between two groups}, an (ordered) cluster of $\A$ determines the (labeled) seed. For an ordered cluster $\x$ of $\A$, we denote by $\S(\x)$ ($B(\x)$, $Q(\x)$ respectively) the labeled seed (exchange matrix, quiver respectively) determined by $\x$.
Now, let $\S=(\x,B)$ be an initial labeled seed of $\A$. We always assume that $\S$ is bipartite, and the Cartan part $A(B)$ of $B$ is a finite type Cartan matrix.
By $Q$ we denote the valued quiver of $B$. Since $\S$ is bipartite, the orientation of $Q$ is unique up to exchanging the sources and the sinks. Recall that $\Phi_{\geq -1}$ is the almost positive root system of $A(B)$, $D_\tau$ is the $\tau$-transform group of $\Phi_{\geq -1}$. Then under correspondence (\ref{equ:cluster variables and roots}), the cluster $\x$ is indexed by negative simple roots in $\Phi_{\geq -1}$.

\begin{thm}\label{thm:tau induce auto}
The transform $\tau_{\pm}$ induces an inverse cluster automorphism of $\A$. Specifically speaking, $f_{\pm}:x_\alpha \mapsto x_{\tau_{\pm}(\alpha)}$ induces an algebra homomorphism of $\A$, such that $f(\x)$ is a cluster of $\A$ with $B(f(\x))=-B$.
\end{thm}
\begin{proof}
We only consider the case of $\tau_-$, similar proof can be made for $\tau_+$.
Firstly we notice that on the bipartite belt defined in \ref{def: bipartite belt}, the cluster variables satisfy the following relations:
\begin{align}
\label{eq:claster-variables on blet+}
x_{i;2r+1}=x_{i;2r} \qquad\, \text{for $\varepsilon(i) = 1,~r\in \Z$;}\\
\label{eq:claster-variables on blet-}
x_{j;2r-1}=x_{i;2r} \qquad\, \text{for $\varepsilon(j) = -1,~r\in \Z$.}
\end{align}
Thus $\x_{2r}=\{x_{i;2r},x_{j;2r-1}~|~\varepsilon(i) = 1, \varepsilon(j) = -1\}$, and $\Y=\bigcup_{r\in \Z}\x_{2r}$ is the set of all the cluster variables on the bipartite belt. We define a map $f_-$ on $\x=\{x_{1;0},\cdots,x_{n;0}\}$ by
\begin{equation}
\label{defn:f-}
f_-(x_{i;0})  =
\begin{cases}
x_{i;0} & \text{if $\varepsilon(i) = 1$;} \\ 
\mu_-(x_{i;0}) & \text{if $\varepsilon(i) = -1$.}
\end{cases}
\end{equation}
Then $f(\x)$ is the cluster of the labeled seed $\mu_-(\S)$, and $B(f_-(\x))= -B(\x)$. Thus from Lemma \ref{lem:equivalent discription of clus-auto}, $f_-$ induces an inverse cluster automorphism of $\A$, which we denote also by $f_-$. Now we show that $f_-$ is in fact induced by $\tau_-$, or equivalently, for any cluster variable $x_\alpha \in \A$, we have
\begin{equation}
\label{eq:f- and tau-}
f_-(x_\alpha)  = x_{\tau_-(\alpha).}
\end{equation}
Due to (\ref{sqcup of d-vectors}) and \ref{d-vectors on bipartite belt}, the set of all cluster variables lies in the bipartite belt, that is, $\Y$ contains all the cluster variables of $\A$. Therefore we only need to prove (\ref{eq:f- and tau-}) for each element in $\Y$. On the one hand, since $f_-$ is a cluster automorphism, it induces an automorphism of $E_\A$ from Lemma \ref{lem:relations between two groups}(2). Specially, by viewing the bipartite belt as a subgraph of $E_\A$, $f_-$ induces an automorphism of the bipartite belt. By comparing the cluster variables, it is easy to see that this automorphism is in fact a reflection of the bipartite belt, such that $f_-(\S_{2r})=\S_{-2r+1}$ and $f_-(\S_{-2r+1})=\S_{2r}$ for $r \geqslant 0$ (for example, see Example \ref{exm:A3 type exchange graph} for the case of type $A_3$, see Example \ref{exm:B3 C3 type exchange graph} for the cases of type $B_3$ and type $C_3$). Under this automorphism, $f_-$ induces a permutation of $\Y$ as follows:
\begin{equation}
\label{correp:f-1}
f_-:
\begin{cases}
x_{i;-2r+1} \leftrightarrow x_{i;2r} & \text{if $\varepsilon(i) = 1$}, r \geqslant 0; \\ 
x_{i;-2r+1} \leftrightarrow x_{i;2r} & \text{if $\varepsilon(i) = -1$}, r \geqslant 0.
\end{cases}
\end{equation}
By using equalities (\ref{eq:claster-variables on blet+}) and (\ref{eq:claster-variables on blet-}), the above permutation is equivalent to the following one:
\begin{equation}
\label{correp:f-2}
f_-:
\begin{cases}
x_{i;-2r} \leftrightarrow x_{i;2r} & \text{if $\varepsilon(i) = 1$}, r \geqslant 0; \\ 
x_{i;-2r+1} \leftrightarrow x_{i;2r-1} & \text{if $\varepsilon(i) = -1$}, r \geqslant 0.
\end{cases}
\end{equation}
On the other hand, by Definition \ref{def:d-vectors}, $\tau_-$ induces a permutation on $\Phi_{\geq -1}$:
\begin{equation}
\label{correp:tau-}
\tau_-:
\begin{cases}
\b(i;-2r) \leftrightarrow \b(i;2r) & \text{if $\varepsilon(i) = 1$}, r \geqslant 0; \\ 
\b(i;-2r+1) \leftrightarrow \b(i;2r-1) & \text{if $\varepsilon(i) = -1$}, r \geqslant 0.
\end{cases}
\end{equation}
Then by the equality
\[
x_{i;r}=\frac{P_{i;r}(\x)}{\x^{{\bf{d}}(i;r)}}
\]
given in Lemma \ref{d-vectors on bipartite belt}, we know that the permutation of $f_-$ on $\Y$ coincides with the permutation of $\tau_-$ on $\Phi_{\geq -1}$. Therefore we have equality (\ref{eq:f- and tau-}) for any cluster variables of $\A$. Thus $f_-$ is induced by $\tau_-$.
\end{proof}

\begin{cor}
\label{cor:tau group is subgp of auto group}
$D_{\tau}\subseteq Aut(\A)$.
\end{cor}
\begin{proof}
It follows from Theorem \ref{thm:tau induce auto} that $\tau_\pm$ induces 
a cluster automorphism $f_\pm$ in $Aut(\A)$. To prove that $D_{\tau}$ is 
a subgroup of $Aut(\A)$, we show that $D_{\tau}\cong<f_-,f_+>$. Note that a cluster automorphism in $<f_-,f_+>$ is an identity if and only if its action is an identity, if and only if the corresponding element in $D_\tau$ is an identity. Thus we know that the relations in $<f_-,f_+>$ correspond to the relations
in $<\tau_-,\tau_+>$. Therefore $D_{\tau}\cong <f_-,f_+>\subseteq Aut(\A)$.
\end{proof}
We can see from Lemma \ref{lem:structure of tau group} (and table \ref{table:tau group}) that $D_\tau$ is a dihedral group with two generators, one is a reflection $\tau_-$ (or $\tau_+$), the other one is a rotation $\tau_-\tau_+$. Then the order of $\tau_-\tau_+$ determines $D_\tau$. In order to find the relations between $D_\tau$ and $Aut(\A)$, we consider the corresponding element $f_0:=f_-f_+$ in $Aut(\A)$. If $w_0=-1$, then $(\tau_-\tau_+)^{\frac{h+2}{2}}=1$, and thus $f_0=(f_-f_+)^{\frac{h+2}{2}}$ is an identity. If $w_0\neq -1$, then $(\tau_-\tau_+)^{h+2}=1$, while $(\tau_-\tau_+)^{\frac{h+2}{2}}\neq 1$ ($h\in 2\Z$) or $\tau_+(\tau_-\tau_+)^{\frac{h+1}{2}}\neq 1$ ($h\in 2\Z+1$). Thus $f_0=(f_-f_+)^{\frac{h+2}{2}}$ or $f_0=f_+(f_-f_+)^{\frac{h+1}{2}}$ is not an identity, and by Theorem \ref{thm:tau induce auto} and equality (\ref{eq:d-vector under longest element}), the images of initial cluster variables in $\x$ are given by the action of longest element $w_0$ in Weyl group:
\begin{equation}
\label{eq: action f0}
f_0(x_{-\alpha_i})=
\begin{cases}
x_{-\alpha_{i^*}} & \text{if $\varepsilon(i) = (-1)^h$;} \\ 
x_{-\alpha_{i^*}} & \text{if $\varepsilon(i) = (-1)^{h-1}$.}
\end{cases}
\end{equation}
Therefore $f_0$ maps the initial seed to itself, and induces an automorphism of the initial quiver. We now consider the specific action of $f_0$ on $\A$ of a given Dynkin type.\\

{\bf{Case I}}, type $A_{2n} (n\geqslant 1)$. The underlying graph of $Q$ is:
\begin{center}
\begin{tikzpicture}
                        \node[] () at (-2.5,0) {$1$};
                        \node[] () at (-1.5,0) {$2$};
                        \node[] () at (0,-0.05) {$\cdots$};
                        \node[] () at (1.8,0) {$2n-1$};
                        \node[] () at (3.3,0) {$2n$};
					    \draw[-,thick] (-2.3,0) -- (-1.7,0);
					    \draw[-,thick] (-1.3,0) -- (-0.7,0);
					    \draw[-,thick] (0.5,0) -- (1.1,0);
					    \draw[-,thick] (2.4,0) -- (3,0);
\end{tikzpicture}
\end{center}
This is the only Dynkin type such that $h\in 2\Z+1$. Here $w_0\neq -1$, and it induces a permutation of $-\Pi$:
\begin{equation*}
\begin{array}{ccc}
-\alpha_{1} & \leftrightarrow &-\alpha_{2n}\\
-\alpha_{2} & \leftrightarrow &-\alpha_{2n-1}\\
&\cdots&\\
-\alpha_{n} & \leftrightarrow &-\alpha_{n+1}
\end{array}
\end{equation*}
This permutation corresponds to a permutation of initial cluster variables given by $f_0=f_+(f_-f_+)^{\frac{h+1}{2}}$, which induces a cluster automorphism. Note that this cluster automorphism makes a reflection on $Q$, exchanging the sources and the sinks, thus it is an inverse cluster automorphism.\\

{\bf{Case II}}, type $A_{2n+1} (n\geqslant 1)$:
\begin{center}
\begin{tikzpicture}
                        \node[] () at (-2.5,0) {$1$};
                        \node[] () at (-1.5,0) {$2$};
                        \node[] () at (0,-0.05) {$\cdots$};
                        \node[] () at (1.5,0) {$2n$};
                        \node[] () at (3.3,0) {$2n+1$};
					    \draw[-,thick] (-2.3,0) -- (-1.7,0);
					    \draw[-,thick] (-1.3,0) -- (-0.7,0);
					    \draw[-,thick] (0.5,0) -- (1.1,0);
					    \draw[-,thick] (1.9,0) -- (2.5,0);
\end{tikzpicture}
\end{center}
In this case, $w_0\neq -1$, it induces a permutation on $-\Pi$ by:
\begin{equation*}
\begin{array}{ccc}
-\alpha_{1} & \leftrightarrow & -\alpha_{2n+1}\\
-\alpha_{2} & \leftrightarrow & -\alpha_{2n}\\
&\cdots &\\
-\alpha_{n+1} & \leftrightarrow & -\alpha_{n+1}
\end{array}
\end{equation*}
This permutation corresponds to a permutation of initial cluster variables given by $f_0=f_+(f_-f_+)^{\frac{h+1}{2}}$, which reflects $Q$ at the central point of $Q$. Notice that this reflection maintains the sources and the sinks of $Q$. Therefore $f_0$ is a direct cluster automorphism.\\

{\bf{Case III}}, type $D_{2n} (n\geqslant 2)$:
\begin{center}
\begin{tikzpicture}
                        \node[] () at (-2.5,0) {$1$};
                        \node[] () at (-1.5,0) {$2$};
                        \node[] () at (0,-0.05) {$\cdots$};
                        \node[] () at (1.8,0) {$2n-2$};
                        \node[] () at (3.5,0.7) {$2n-1$};
                        \node[] () at (3.5,-0.7) {$2n$};
					    \draw[-,thick] (-2.3,0) -- (-1.7,0);
					    \draw[-,thick] (-1.3,0) -- (-0.7,0);
					    \draw[-,thick] (0.5,0) -- (1.1,0);
					    \draw[-,thick] (2.5,0.1) -- (3,0.4);
					    \draw[-,thick] (2.5,-0.1) -- (3,-0.4);
\end{tikzpicture}
\end{center}
In this case, exchanging initial cluster variables $x_{-\alpha_{2n-1}}$ and $x_{-\alpha_{2n}}$ induces a direct cluster automorphism $f'$ of $A_Q$, where:
\begin{equation}
\label{eq: f'}
f'(x_{-\alpha_i})=
\begin{cases}
x_{-\alpha_{i-1}} & \text{if $i=2n$;} \\ 
x_{-\alpha_{i+1}} & \text{if $i=2n-1$;}\\ 
x_{-\alpha_{i}} & \text{otherwise.}
\end{cases}
\end{equation}
Notice that $w_0= -1$, thus by Lemma \ref{lem:tau group}, the orbits of $-\alpha_{2n-1}$ and $-\alpha_{2n}$ under the action of $D_{\tau}$ are different. Therefore $f'$ does not belong to $D_{\tau}$. Moreover, it is clear that $f'$ commutes with elements in $D_{\tau}$, thus
\begin{equation}
D_{\tau}\times <f'> \subseteq Aut(\A),
\end{equation}
particularly,
\begin{equation}
\begin{cases}
D_{\tau}\times S_3 \subseteq Aut(\A) & \text{if $n= 2$;} \\ 
D_{\tau}\times \Z_2 \subseteq Aut(\A) & \text{if $n\neq 2$.}
\end{cases}
\end{equation}\\

{\bf{Case IV}}, type $D_{2n+1} (n\geqslant 2)$:
\begin{center}
\begin{tikzpicture}
                        \node[] () at (-2.5,0) {$1$};
                        \node[] () at (-1.5,0) {$2$};
                        \node[] () at (0,-0.05) {$\cdots$};
                        \node[] () at (1.8,0) {$2n-1$};
                        \node[] () at (3.5,0.7) {$2n+1$};
                        \node[] () at (3.5,-0.7) {$2n$};
					    \draw[-,thick] (-2.3,0) -- (-1.7,0);
					    \draw[-,thick] (-1.3,0) -- (-0.7,0);
					    \draw[-,thick] (0.5,0) -- (1.1,0);
					    \draw[-,thick] (2.5,0.1) -- (3,0.4);
					    \draw[-,thick] (2.5,-0.1) -- (3,-0.4);
\end{tikzpicture}
\end{center}
The longest element $w_0\neq -1$ induces a permutation on $-\Pi$ by:
\begin{equation*}
\begin{array}{ccc}
-\alpha_{1}& \leftrightarrow&-\alpha_{1}\\
-\alpha_{2}&\leftrightarrow&-\alpha_{2}\\
&\cdots &\\
-\alpha_{2n-1}&\leftrightarrow&-\alpha_{2n-1}\\
-\alpha_{2n}&\leftrightarrow&-\alpha_{2n+1}
\end{array}
\end{equation*}
Then $f_0$ is a direct automorphism of $\A$, with exchanging cluster variables $x_{-\alpha_{2n}}$ and $x_{-\alpha_{2n+1}}$.\\

{\bf{Case V}}, type $E_6$:
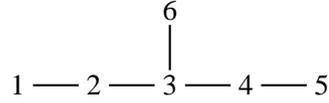
\begin{center}
\begin{tikzpicture}
                        \node[] () at (-2,0) {$1$};
                        \node[] () at (-1,0) {$2$};
                        \node[] () at (0,0) {$3$};
                        \node[] () at (1,0) {$4$};
                        \node[] () at (2,0) {$5$};
                        \node[] () at (0,1) {$6$};
					    \draw[-,thick] (-1.8,0) -- (-1.2,0);
					    \draw[-,thick] (-0.8,0) -- (-0.2,0);
					    \draw[-,thick] (1.8,0) -- (1.2,0);
					    \draw[-,thick] (0.8,0) -- (0.2,0);
					    \draw[-,thick] (0,0.2) -- (0,0.8);
\end{tikzpicture}
\end{center}
In this case, $w_0\neq -1$ induces a permutation of $-\Pi$:
\begin{equation*}
\begin{array}{ccc}
-\alpha_{1}&\leftrightarrow&-\alpha_{5}\\
-\alpha_{2}&\leftrightarrow&-\alpha_{4}\\
-\alpha_{3}&\leftrightarrow&-\alpha_{3}\\
-\alpha_{6}&\leftrightarrow&-\alpha_{6}
\end{array}
\end{equation*}
The corresponding cluster automorphism $f_0$ is direct, and the corresponding automorphism of $Q$ is a reflection with respect to the center axis of $Q$.\\

{\bf{Case VI}}, types $E_7,E_8,B_n,C_n,F_4,G_2$: In these cases, $w_0= -1$, $f_0=1$, and the automorphism group of $Q$ is trivial.\\

Finally, from above discussions and by comparing table \ref{table:simply laced auto group} of cluster automorphism groups of simply laced finite type cluster algebras and table \ref{table:tau group} of $\tau$-transform groups, we have the following corollary. It is worth to note that for the even integer $2(2n+1)$, we have an isomorphism of dihedral groups: $D_{2(2n+1)}\cong D_{2n+1}\times \Z_2$.
\begin{cor}\label{cor:simply laced auto group and tau group}
\begin{enumerate}
\item For cluster algebras of types $A_n(n\geqslant 2),D_{2n+1}(n\geqslant 2),E_6,E_7,E_8$, we have
$$Aut(\A)\cong D_{\tau};$$
\item for cluster algebras of type $D_{2n}(n\neq 2)$, we have
$$Aut(\A)\cong D_{\tau}\times \Z_2;$$
\item for cluster algebra of type $D_{4}$, we have
$$Aut(\A)\cong D_{\tau}\times S_3.$$
\end{enumerate}
\end{cor}
We will use the same notation $D_{\tau}$ to denote the group $ <f_-,f_+>$. Then any element $\sigma$ in $D_{\tau}$ maps $\S$ to a bipartite seed of $\A$, with the exchange matrix isomorphic to $B$ or $-B$. We denote it by $\sigma(\S)$.

\begin{exm}\label{exm:A3 type exchange graph}
We consider the cluster algebra $\A$ of type $A_3$ with initial labeled seed $\S=(\{x_1,x_2,x_3\}, Q)$, where $Q$ is
$\xymatrix{1\ar[r]&2&3\ar[l]}.$ Then its exchange graph $E_{\A}$ is depicted in Figure \ref{automorphisms of exchange graph of type $A_3$}. Note that there are three quadrilaterals and six pentagons in $E_{\A}$. The seeds $\S, \S_1, \S_2, \S_3, \S_4, \S_5$ in the bipartite belt (\ref{eq:big-circle-seeds+}) are depicted in the figure.
By Corollary \ref{cor:simply laced auto group and tau group} and Lemma \ref{lem:relations between two groups}, $D_{\tau}\cong D_6\subseteq Aut(E_\A)$.
The cluster automorphism $f_-\in D_{\tau}$ is defined by:
\begin{equation*}
f_-: \begin{cases}
x_1 \mapsto x_1 \\ 
x_2 \mapsto \mu_2(x_2) \\ 
x_3 \mapsto x_3
\end{cases}
\end{equation*}
It maps $\S$ to $\S_1$, and induces a reflection with respect to the horizontal central axis of $E_\A$. The cluster automorphism $f_+\in D_{\tau}$ maps $\S$ to $\S_5$, and also induces an automorphism of $E_\A$. The cluster automorphism $f_0=f_+f_-$ is defined by:
\begin{equation*}
f_0: \begin{cases}
x_1 \mapsto x_3 \\ 
x_2 \mapsto x_2 \\ 
x_3 \mapsto x_1
\end{cases}
\end{equation*}
It maintains the bipartite seeds of $\A$, and induces a reflection with respect to the vertical central axis of $E_\A$. The group $Aut(\A)\cong D_6$ can be viewed as the symmetry group of the bipartite belt consisting of $\{\S, \S_1, \S_2, \S_3, \S_4, \S_5\}$.  In fact the automorphisms induced by $f_-$ and $f_+$ generate all the automorphism of $E_\A$, that is, we have $Aut(\A)\cong Aut(E_\A)$, see Example 3 in \cite{CZ15b}.

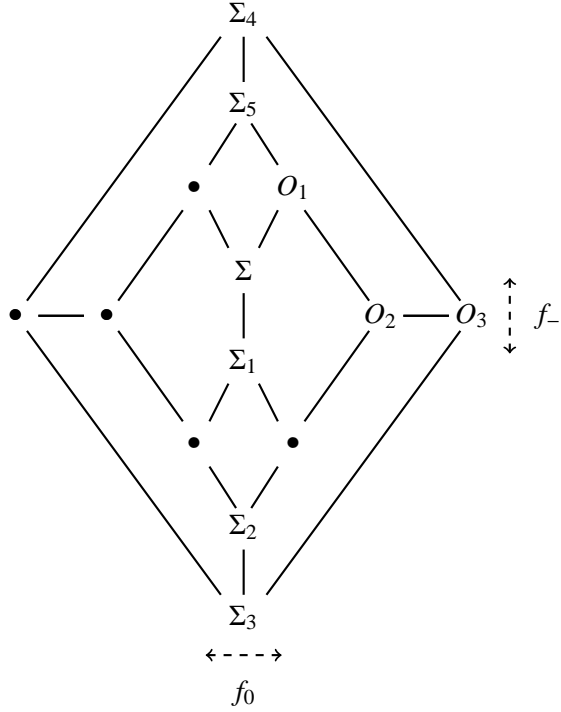
\begin{figure}
\begin{center}
{\begin{tikzpicture}

\node[] (C) at (0,4)
						{$\S_4$};
\node[] (C) at (0,2.8)
						{$\S_5$};
\node[] (C) at (0,0.6)
						{$\S$};
\node[] (C) at (3,0)
						{$O_3$};
\node[] (C) at (1.8,0)
						{$O_2$};
\node[] (C) at (0.65,1.7)
						{$O_1$};
\draw[thick] (0,3.7) -- (0,3.1);
\draw[thick] (2.1,0) -- (2.7,0);
\draw[thick] (0.3,3.7) -- (2.85,0.2);
\draw[thick] (0.1,2.55) -- (0.45,2.0);
\draw[thick] (0.2,0.9) -- (0.45,1.4);
\draw[thick] (1.65,0.2) -- (0.8,1.4);

\node[] (C) at (-3,0)
						{$\bullet$};
\node[] (C) at (-1.8,0)
						{$\bullet$};
\node[] (C) at (-0.65,1.7)
						{$\bullet$};
\draw[thick] (-2.1,0) -- (-2.7,0);
\draw[thick] (-0.3,3.7) -- (-2.85,0.2);
\draw[thick] (-0.1,2.55) -- (-0.45,2.0);
\draw[thick] (-0.2,0.9) -- (-0.45,1.4);
\draw[thick] (-1.65,0.2) -- (-0.8,1.4);


\node[] (C) at (0,-4)
						{$\S_3$};
\node[] (C) at (0,-2.8)
						{$\S_2$};
\node[] (C) at (0,-0.6)
						{$\S_1$};

\node[] (C) at (0.65,-1.7)
						{$\bullet$};
\draw[thick] (0,-3.7) -- (0,-3.1);
\draw[thick] (2.1,0) -- (2.7,0);
\draw[thick] (0.3,-3.7) -- (2.85,-0.2);
\draw[thick] (0.1,-2.55) -- (0.45,-2.0);
\draw[thick] (0.2,-0.9) -- (0.45,-1.4);
\draw[thick] (1.65,-0.2) -- (0.8,-1.4);

\node[] (C) at (-3,0)
						{$\bullet$};
\node[] (C) at (-1.8,0)
						{$\bullet$};
\node[] (C) at (-0.65,-1.7)
						{$\bullet$};
\node[] (C) at (0,-5)
						{$f_0$};
\node[] (C) at (4,0)
						{$f_-$};
\draw[thick] (-2.1,0) -- (-2.7,0);
\draw[thick] (-0.3,-3.7) -- (-2.85,-0.2);
\draw[thick] (-0.1,-2.55) -- (-0.45,-2.0);
\draw[thick] (-0.2,-0.9) -- (-0.45,-1.4);
\draw[thick] (-1.65,-0.2) -- (-0.8,-1.4);

\draw[thick] (0,-0.3) -- (0,0.3);
\draw[thick,dashed,<->] (-0.5,-4.5) -- (0.5,-4.5);
\draw[thick,dashed,<->] (3.5,0.5) -- (3.5,-0.5);
\end{tikzpicture}}
\end{center}
\begin{center}
\caption{The exchange graph of a cluster algebra of $A_3$ type}
\label{automorphisms of exchange graph of type $A_3$}
\end{center}
\end{figure}
\end{exm}

\subsection{Cluster automorphism groups: non-simply laced cases}\label{Cluster automorphism groups: non-simply laced cases}
In the this subsection, we will firstly recall the folding of a root system and the folding of a cluster algebra from \cite{FZ03a,D08}, and then use them to compute the cluster automorphism group of a non-simply laced finite type cluster algebra.\\

Let $\tilde{\S}=(\tilde{\x},\tilde{B})$ be a simply laced finite type seed, which $\t{\S}$ defines a cluster algebra $\t{\A}$. We assume that $\t{\S}$ is a bipartite seed with $A(\t{B})$ of finite type. Denote by $\tilde{\Pi}$ and $\tilde{\Phi}$ the set of simple roots and root system of $A(\t{B})$ respectively. The quiver of $\t{B}$ is $\tilde{Q}$, with a vertex set $\tilde{{\bf{I}}}$. We write $\t{{\bf{I}}}_+$ and $\t{{\bf{I}}}_-$ to the set of sources and the set of sinks of $\tilde{Q}$ respectively. Denote by $G(\tilde{Q})$ the group of (direct) automorphisms of $\tilde{Q}$, then every automorphism in $G(\tilde{Q})$ preserves the sets $\t{{\bf{I}}}_+$ and $\t{{\bf{I}}}_-$. We list these groups in table \ref{table:simply laced quiver auto group}:\\
\begin{table}[ht]
\begin{equation*}
\begin{array}{cc}
\textrm{Dynkin type}  & \textrm{Automorphism group $G(\tilde{Q})$}\\
\hline
A_{2n} (n\geqslant 1) & \{id\}\\
A_{2n+1} (n\geqslant 1) & \Z_2\\
D_4 & S_3\\
D_n (n\geqslant 5) & \Z_2\\
E_6 & \Z_2\\
E_7 & \{id\}\\
E_8 & \{id\}\\
\end{array}
\end{equation*}
\smallskip
\caption{Automorphism groups of bipartite quivers of simply laced finite type}
\label{table:simply laced quiver auto group}
\end{table}

Let ${\bf{I}}=\tilde{{\bf{I}}}/G(\tilde{Q})$ be the orbit set of $G(\t{Q})$ on $\t{{\bf{I}}}$. Denote by $\pi: \t{{\bf{I}}}\to {\bf{I}}$ the canonical projection, denote also by $\pi$ the projection of Laurent polynomial rings:
\begin{equation}\label{map:pi on polynomials}
\begin{array}{ccc}
\Z[x^{\pm}_{\t{i}}:\t{i}\in \t{{\bf{I}}}] & \rightarrow & \Z[x^{\pm}_{i}:i\in {\bf{I}}]\\
x^{\pm}_{\t{i}}& \rightarrow & x^{\pm}_{\pi(\t{i})}
\end{array}
\end{equation}
An element $g\in G(\t{Q})$ defines an automorphism on $\Z[x^{\pm}_{\t{i}}:\t{i}\in \t{{\bf{I}}}]$ by:
\begin{equation}\label{action on polynomials}
gx^{\pm}_{\t{i}}=x^{\pm}_{g\t{i}}
\end{equation}
for any $\t{i}\in \t{{\bf{I}}}$.
The quotient matrix $B=\pi(\t{B})=\t{B}/G(\t{Q})=(b_{ij})_{{\bf{I}}\times {\bf{I}}}$ is defined, for any $i,j \in {\bf{I}}$, by:

\begin{align}
\label{map:pi on matrix}
b_{ij}=\sum_{\t{k}\in i} a_{\t{k}\t{j}}~,
\end{align}
where $\t{j}$ is an element in $j$, since $G(\tilde{Q})$ is an automorphism group of $\t{Q}$, the definition is independent on the choice of $\t{j}$. Then $B$ is a symmetrizable Cartan matrix of finite type, and table \ref{table:simply laced to valued} shows the specific correspondences. We call above process a folding of a matrix\cite{D08}. For those simply laced Dynkin types which do not appear in the table, their automorphisms are trivial, thus the foldings are also trivial. Denote by $Q,\Phi,\Pi$ and $\S=(\x,B)$ the (valued) quiver, the root system, the set of simple roots, and the seed corresponding to $B$ respectively. The seed $\S$ defines a cluster algebra $\A$. By linearity, we extend the surjective map
${\t{\Pi}}^\vee\to \Pi^\vee: {\alpha_{\t{i}}}^\vee \mapsto {\alpha_{i}}^\vee$ to a surjection ${\t{\Phi}}^\vee\to \Phi^\vee$, and denote it by $\pi$. Then we have surjections $\pi: {\t{\Phi}}\to \Phi : (\pi(\t{\alpha}))^\vee\mapsto \pi({\t{\alpha}}^\vee)$ and $\pi: {\t{\Phi}}_{\geqslant 1}\to \Phi_{\geqslant 1}$. Moreover, $\pi$ and $\tau$ are commute: $\pi\circ {\tau}_{\pm}={\tau}_{\pm}\circ \pi$ (see in \cite{FZ03a}).\\

\begin{table}[ht]
\begin{equation*}
\begin{array}{cc}
\textrm{Dynkin type of $\t{B}$} & \textrm{Dynkin type of $B$} \\
\hline
A_{2n-1} (n\geqslant 2)& B_n \\
D_{n+1} (n\geqslant 4) & C_n\\
E_6 & F_4\\
D_4 & G_2\\
\end{array}
\end{equation*}
\smallskip
\caption{Folding correspondences of Dynkin types}
\label{table:simply laced to valued}
\end{table}

We say that a seed $\S'=(\x',B')$ of $\t{\A}$ is $G(\t{Q})$-invariant, if for any $\t{i}\in \t{{\bf{I}}}$ and $g\in G(\t{Q})$, we have $gx'_{\t{i}}=x'_{g\t{i}}$, where $gx'_{\t{i}}$ is defined in (\ref{action on polynomials}) (this definition is for quivers of Dynkin type, for the general case, we refer to \cite{D08}). For a $G(\t{Q})$-invariant seed $\S'=(\x',B')$, we define a seed $\pi(\S')=(\pi(\x'),\pi(B'))$, where $\pi(\x')$ and $\pi(B')$ are given by (\ref{map:pi on polynomials}) and (\ref{map:pi on matrix}) respectively. In particular, the initial seed $\t{\S}$ is $G(\t{Q})$-invariant, and $\pi(\t{\S})=\S$.

\begin{lem}\cite{D08}\label{lem:folding}
The map $\pi$ induces a bijection between the $G(\t{Q})$-invariant seeds of $\t{\A}$ and seeds of $\A$, particularly,
\begin{align}
\pi(\X(\t{\A}))=\X(\A);\\
\pi(\t{\A})=\A.~~~~~~~
\end{align}
\end{lem}

\begin{thm}\label{thm:tau gp eq to auto gp for nonsimply laced}
We use above notations, where $\A$ is a non-simply laced finite type cluster algebra and $\t{\A}$ is the corresponding simply laced cluster algebra.
\begin{enumerate}
\item An automorphism of $\t{Q}$ induces an automorphism of $\A_{\t{Q}}$, thus we view $G(\t{Q})$ as a subgroup of $Aut(\A_{\t{Q}})$. Then $G(\t{Q})$ is a normal subgroup of $Aut(\A_{\t{Q}})$.
\item $Aut(\A)\cong Aut(\t{\A})/G(\t{Q})$.
\item $Aut(\A)$ is isomorphic to its $\tau$-transform group: $Aut(\A)\cong D_{\tau}$.
\end{enumerate}
\end{thm}
\begin{proof}
\begin{enumerate}
\item It is clear that an automorphism of $\t{Q}$ induces an automorphism of $\A_{\t{Q}}$, now we show that $G(\t{Q})$ is normal. If $\t{Q}$ is of type $A_{2n+1} (n\geqslant 1)$, type $D_{2n+1} (n\geqslant 2)$ or type $E_6$, then $G(\t{Q})\cong\Z_2$. In fact from discussions in {\bf{Case II}}, {\bf{Case IV}} and {\bf{Case V}} respectively, the non-trivial automorphism of $\t{Q}$ is induced by $f_0$. Note that in these cases, $<f_0>\cong \Z_2$ is a normal subgroup of $D_\tau$. Thus $G(\t{Q})$ is a normal subgroup of $Aut(\A_{\t{Q}})\cong D_\tau$. If $\t{Q}$ is of type $D_{2n} (n\geqslant 2)$, then from {\bf{Case III}}, $G(\t{Q})\cong \Z_2\cong <f_0>$ is the second part in the direct product $Aut(\A_{\t{Q}})\cong D_\tau\times \Z_2$, thus $G(\t{Q})$ is a normal subgroup of $Aut(\A_{\t{Q}})$. Similarly, if $\t{Q}$ is of type $D_{4}$, $G(\t{Q})\cong S_3$ is a normal subgroup of $Aut(\A_{\t{Q}})\cong D_\tau\times S_3$.
    For the cases of $E_7$ and $E_8$, $G(\t{Q})$ is trivial. Thus we also have the conclusion.
\item
To prove the statement, we only need to notice the following two observations: On the one hand, every seed $\t{\S}'$ of $\t{\A}$, with quiver isomorphic to $\t{Q}$ or $\t{Q}^{op}$, is $G(\t{Q})$-invariant, and the quiver of $\pi({\t{\S}}')$ is isomorphic to $Q$ or $Q^{op}$ respectively. On the other hand, the preimage of a seed of $\A$, with quiver isomorphic to $Q$ or $Q^{op}$, is a seed of $\t{\A}$ with quiver isomorphic to $\t{Q}$ or $\t{Q}^{op}$ respectively. Then Lemma \ref{lem:equivalent discription of clus-auto} yields that a cluster automorphism of $\t{\A}$ induces a cluster automorphism of $\A$, and conversely any cluster automorphism of $\A$ is induced from a cluster automorphism of $\t{\A}$. Moreover, it is easy to see that a cluster automorphism of $\A$ is trivial if and only if it is induced from an element in $G(\t{Q})\subset Aut(\t{\A})$. Therefore we have $Aut(\A)\cong Aut(\t{\A})/G(\t{Q})$.
\item
By comparing the group $Aut(\t{\A})/G(\t{Q})$ and the $\tau$-transform group $D_\tau$ of $\A$ in table \ref{table:tau group}, this follows from the second statement.
\end{enumerate}
\end{proof}

\begin{figure}
\begin{center}
{\begin{tikzpicture}[scale=0.8]

\node[] (C) at (0,6)
						{$O_1$};

\node[] (C) at (1,0)
						{$O_4$};
\node[] (C) at (2,1.5)
						{$O_3$};
\node[] (C) at (1,3)
						{$\S$};
\node[] (C) at (3,3)
						{$\S_7$};
\node[] (C) at (2,4.5)
						{$O_2$};

\draw[thick] (-0.8,0) -- (0.8,0);
\draw[thick] (-0.8,3) -- (0.8,3);

\draw[thick] (1.2,0.2) -- (1.8,1.3);
\draw[thick] (1.2,2.8) -- (1.8,1.7);
\draw[thick] (1.2,3.2) -- (1.8,4.3);
\draw[thick] (2.8,3.2) -- (2.2,4.3);
\draw[thick] (1.8,4.7) -- (0.2,5.8);
\draw[thick] (2.2,1.7) -- (2.8,2.8);

\node[] (C) at (-1,0)
						{$\bullet$};
\node[] (C) at (-2,1.5)
						{$\bullet$};
\node[] (C) at (-1,3)
						{$\S_1$};
\node[] (C) at (-3,3)
						{$\S_2$};
\node[] (C) at (-2,4.5)
						{$\bullet$};

\draw[thick] (-1.2,0.2) -- (-1.8,1.3);
\draw[thick] (-1.2,2.8) -- (-1.8,1.7);
\draw[thick] (-1.2,3.2) -- (-1.8,4.3);
\draw[thick] (-2.8,3.2) -- (-2.2,4.3);
\draw[thick] (-1.8,4.7) -- (-0.2,5.8);
\draw[thick] (-2.2,1.7) -- (-2.8,2.8);
\node[] (C) at (0,-6)
						{$\bullet$};
\node[] (C) at (2,-1.5)
						{$\bullet$};
\node[] (C) at (1,-3)
						{$\S_5$};
\node[] (C) at (3,-3)
						{$\S_6$};
\node[] (C) at (2,-4.5)
						{$\bullet$};

\draw[thick] (-0.8,-3) -- (0.8,-3);

\draw[thick] (1.2,-0.2) -- (1.8,-1.3);
\draw[thick] (1.2,-2.8) -- (1.8,-1.7);
\draw[thick] (1.2,-3.2) -- (1.8,-4.3);
\draw[thick] (2.8,-3.2) -- (2.2,-4.3);
\draw[thick] (1.8,-4.7) -- (0.2,-5.8);
\draw[thick] (2.2,-1.7) -- (2.8,-2.8);

\node[] (C) at (-2,-1.5)
						{$\bullet$};
\node[] (C) at (-1,-3)
						{$\S_4$};
\node[] (C) at (-3,-3)
						{$\S_3$};
\node[] (C) at (-2,-4.5)
						{$\bullet$};

\draw[thick] (-1.2,-0.2) -- (-1.8,-1.3);
\draw[thick] (-1.2,-2.8) -- (-1.8,-1.7);
\draw[thick] (-1.2,-3.2) -- (-1.8,-4.3);
\draw[thick] (-2.8,-3.2) -- (-2.2,-4.3);
\draw[thick] (-1.8,-4.7) -- (-0.2,-5.8);
\draw[thick] (-2.2,-1.7) -- (-2.8,-2.8);
\draw[thick] (3,2.8) -- (3,-2.8);
\draw[thick] (-3,2.8) -- (-3,-2.8);

\draw[thick] (0.4,5.9) arc (70:-70:6.3);

\end{tikzpicture}}
\end{center}
\begin{center}
\caption{The exchange graph of a cluster algebra of type $B_3$ or type $C_3$}
\label{automorphisms of exchange graph of type $BC_3$}
\end{center}
\end{figure}
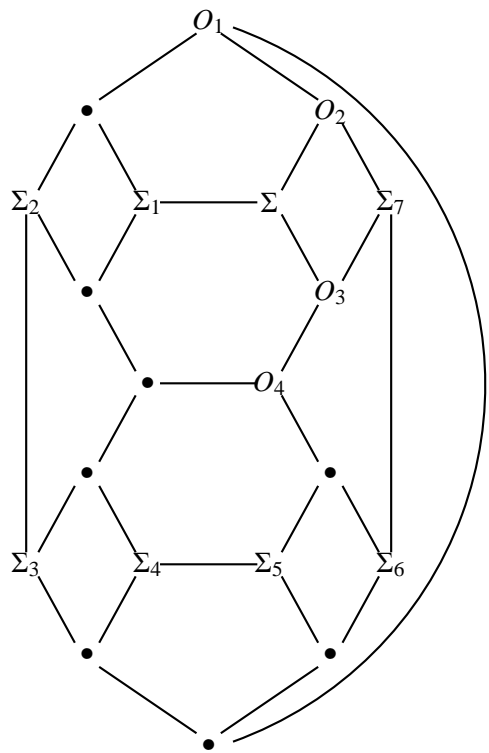
\begin{exm}\label{exm:B3 C3 type exchange graph}
It is well known that the cluster algebras of type $B_n$ and type $C_n$ have the
same combinatorial structure. Originating at a bipartite seed $\S$,
the exchange graph of a cluster algebra $\A$ of type $B_3$ or type $C_3$
is depicted in Figure \ref{automorphisms of exchange graph of type $BC_3$},
where $\S, \S_1, \cdot \cdot \cdot$ are seeds on the bipartite belt (\ref{eq:big-circle-seeds+}).
For the cluster algebra of type $B_3$, the quiver of the initial seed $\S$ is \begin{center}
{\begin{tikzpicture}
\node[] (C) at (-1.5,0)  {$1$};
\node[] (C) at (0,0)  {$2$};	
\node[] (C) at (1.5,0)  {$3.$};
\node[] (C) at (0.7,0.3)  {\qihao{(2,1)}};
\draw[<-,thick] (-0.2,0) -- (-1.3,0);
\draw[<-,thick] (0.2,0) -- (1.3,0);
\end{tikzpicture}}
\end{center}
For the cluster algebra of type $C_3$, the quiver of the initial seed $\S$ is
\begin{center}
{\begin{tikzpicture}
\node[] (C) at (-1.5,0)  {$1$};
\node[] (C) at (0,0)  {$2$};						
\node[] (C) at (1.5,0)  {$3.$};
\node[] (C) at (0.7,0.3)  {\qihao{(1,2)}};
\draw[<-,thick] (-0.2,0) -- (-1.3,0);
\draw[<-,thick] (0.2,0) -- (1.3,0);
\end{tikzpicture}}
\end{center}
Then $Aut(\A)\cong D_{\tau}=<f_-,f_+>\cong D_4$, where $f_-$ maps $\S$ to $\S_1$, $f_+$ maps $\S$ to $\S_7$, however, $f_0=(f_-f_+)^{4}=1$, thus $Aut(\A)$ is not the symmetry group of the bipartite belt.
\end{exm}

\begin{cor}\label{cor:only one bipartite belt}
Let $\A$ be a cluster algebra of finite type. 
\begin{itemize}
\item[1] If a quiver $Q$ of $\A$ is bipartite, then its underlying diagram $\overline{Q}$ must be a Dynkin diagram.
\item[2] All the bipartite seeds are on the bipartite belt in Definition \ref{def: bipartite belt}, more precisely, any bipartite seed is of the form (\ref{eq:big-circle-seeds+}) or of the form (\ref{eq:big-circle-seeds-}).
\item[3] There is only one bipartite belt on $E_\A$.
\end{itemize}
\end{cor}
\begin{proof}
\begin{itemize}
\item[1] If $Q$ is a tree, then from Proposition 9.3 in \cite{FZ03}, $\overline{Q}$ is a Dynkin diagram. Assume that $Q$ has a cycle $Q'$ as a full subquiver. Since $Q$ is bipartite, $Q'$ is a bipartite cycle. Then Proposition 9.7 in \cite{FZ03} shows that $Q'$ is not a quiver of finite type. This contradicts to the assumption that $\A$ is a finite type cluster algebra. Thus $\overline{Q}$ is a Dynkin diagram.
\item[2] Since a Dynkin diagram is a tree, it has only two kinds of bipartite quivers up to isomorphism. Then it follows from the first statement that $\A$ has only two kinds of bipartite quivers up to isomorphism, that is, the quivers in seeds (\ref{eq:big-circle-seeds+}) and (\ref{eq:big-circle-seeds-}). Let $Q$ be a bipartite quiver of $\A$, then Lemma \ref{lem:equivalent discription of clus-auto} yields that there must exist a cluster automorphism of $\A$ which maps $Q$ to a quiver in seeds (\ref{eq:big-circle-seeds+}) and (\ref{eq:big-circle-seeds-}). Note that on the one hand, Lemma \ref{lem:relations between two groups} states that each cluster automorphism of $\A$ gives an automorphism of the exchange graph $E_\A$, on the other hand, as in the proof of Theorem \ref{thm:tau induce auto}, any element in $D_\tau$ induces an automorphism of the bipartite belt in Definition \ref{def: bipartite belt}. Thus Corollary \ref{cor:simply laced auto group and tau group} and Theorem \ref{thm:tau gp eq to auto gp for nonsimply laced} yield that each cluster automorphism of $\A$ gives an automorphism of the bipartite belt in Definition \ref{def: bipartite belt}. Therefore $Q$ must be in the seeds (\ref{eq:big-circle-seeds+}) and (\ref{eq:big-circle-seeds-}). 
\item[3] This follows from the second part.
\end{itemize}
\end{proof}

\begin{rem}
In \cite{ASS12}, the authors compute the automorphism group of an acyclic skew-symmetric cluster algebra by computing the automorphism group of the $AR$-quiver of the cluster category \cite{BMRRT06}, which gives a categorification of the cluster algebra. As it is shown in \cite{BMV10}, one can use tube category $\C$ to categorify cluster algebra $\A$ of type $B_n$ (similar for type $C_n$ in \cite{ZZ14}). Similar to $\cite{ASS12}$, one can compute the cluster automorphism group of $\A$ by computing the automorphism group of the $AR$-quiver of $\C$. We lift the cluster automorphism $f_+f_-$ in $\tau$-transform group $D_{\tau}$ to the $AR$-transform $\tau$ of $\C$, where $\tau$ is an auto-equivalent functor of $\C$. The functor $\tau$ induces a rotation on the $AR$-quiver $Q_{\C}$ of $\C$. Since $Q_{\C}$ is a tube of rank $n+1$, its automorphisms are only rotations and reflections, and in fact they are all induced by elements in $D_{\tau}$. Thus by this way, we also have $Aut(\A)\cong D_{\tau}\cong D_{n+1}$.
\end{rem}

\section{FZ-universal cluster algebras}
\label{Sec universal cluster algebra}
For a cluster algebra with trivial coefficients, its universal cluster algebra is introduced in \cite{FZ07} (Definition 12.3). As a slight change of the universal cluster algebra, a universal geometric cluster algebra is defined in \cite{R14}, it is a geometric cluster algebra. Roughly speaking, a universal cluster algebra is a universal object in the set of cluster algebras with the same principal part, in the view point of coefficient specialization. A coefficient specialization is a ring homomorphism between two cluster algebras which commutes with mutations in both cluster algebras (see Definition 12.1 in \cite{FZ07}), it is a generalization of specialization defined in Definition \ref{def: specialization}. Then a universal cluster algebra $\A^{univ}$ is a cluster algebra such that for each cluster algebra $\A$ with the same principal part as $\A^{univ}$, there exists a coefficient specialization from $\A^{univ}$ to $\A$. Just like any other mathematical object with universal property, a universal cluster algebra is unique in some sense (see section 12 in \cite{FZ07}). Similarly, one can define a universal geometric cluster algebra, which has similar properties as a universal cluster algebra. We refer to \cite{FZ07,R14,CZ15} for the explicit definitions and the properties of these algebras. The following result will be used in the proof of the main theorem of this subsection:
\begin{lem}\label{lem:univ-cluster algebra}
Any universal geometric cluster algebra is gluing free.
\end{lem}
\begin{proof}
This is similar to the proof of Proposition 3.9 in \cite{CZ15}, where this is proved for universal geometric cluster algebras defined by extended skew-symmetric matrices.
The outline of the proof is as follows. Assume that the universal geometric cluster algebra $\A$ is not gluing free. Then, firstly, we divide the frozen cluster variables into disjoint strictly glueable sets. Secondly, we construct a new algebra $\A'$ from $\A$ by specializing frozen cluster variables to $1$, excepting one represent in each strictly glueable set (see the notion of strictly glueable in Definition \ref{def: gluing free labeled seeds}). In fact, $\A'$ is a cluster algebra given by a labeled seed $(\ex,\fx',B')$, where $B'$ is obtained from the exchange matrix $B$ in a fixed labeled seed $(\ex,\fx,B)$ of $\A$ by deleting frozen rows, excepting the represent row in each strictly glueable set, and $\fx'$ is obtained from $\fx$ by deleting the corresponding variables. Finally, one can extend the natural injection from $\ex\sqcup\fx'$ to $\ex\sqcup\fx$ to a coefficient specialization $\A'$ to $\A$. Then $\A'$ is another universal geometric cluster algebra with the same principal part as $\A$. This contradicts to the uniqueness of the universal geometric cluster algebra. Therefore the universal geometric cluster algebra $\A$ is gluing free.
\end{proof}
For a cluster algebra with trivial coefficients, the existence of its universal cluster algebra is not clear at all. However, in \cite{FZ07} (Theorem 12.4), S. Fomin and A. Zelevinsky prove that for a finite type cluster algebra $\A$, there exists a universal cluster algebra $\A^{univ}$, we call it FZ-universal cluster algebra of $\A$. Moreover $\A^{univ}$ is a geometric cluster algebra, and thus a universal geometric cluster algebra. Let $\S=(\x,B)$ be a seed of $\A$, where $B=(b_{ji})_{\bf{I}\times\bf{I}}$ is bipartite and $A(B)$ is of finite type.
Then $\A_{Q^{univ}}$ has a seed $\S^{univ}=(\x^{univ},B^{univ})$ (see section 12 \cite{FZ07}), where $B^{univ}=(\tilde{b}_{ji})_{({\bf{I}}\sqcup\Phi^\vee_{\geq -1})\times\bf{I}}$ is defined by:
\begin{equation}
\label{matrix of universal algebra}
\tilde{b}_{ji}=
\begin{cases}
b_{ji} & \text{if $j\in {\bf{I}}$}; \\ 
\varepsilon(i)[\alpha^\vee:\alpha^\vee_i] & \text{if $j=\alpha^\vee\in \Phi^\vee_{\geq -1}$}.
\end{cases}
\end{equation}

\begin{thm}\label{thm:auto gp of finit type univ alg}
For the finite type cluster algebra $\A$, we have $Aut(\A)\cong Aut(\A^{univ})$.
\end{thm}
\begin{proof}
From Lemma \ref{lem:relations between two groups} and Lemma \ref{lem:univ-cluster algebra}, $Aut(\A^{univ})$ is a subgroup of $Aut(\A)$. Thus we only need to show that $Aut(\A)\subseteq Aut(\A^{univ})$. Firstly, we prove that $D_\tau\subseteq Aut(\A^{univ})$, equivalently, both $f_-$ and $f_+$ induce cluster automorphisms of $\A^{univ}$. We only prove the case of $f_-$, the case of $f_+$ can be proved similarly. Considering the mutation $\mu_-(B^{univ})$, we can see that $\mu_-$ acts on the principal part of $B^{univ}$, for any $i,j\in {\bf{I}}$, by:
\begin{equation}\label{equ:univ1}
\mu_-(\tilde{b}_{ji})=-\tilde{b}_{ji}.
\end{equation}
Since $\S$ is bipartite, by the definition of $\tilde{B}$, we also check that:
\begin{equation}\label{equ:univ}
\begin{array}{lll}
\mu_-(\tilde{b}_{\alpha^\vee j})&=&
\begin{cases}
-\tilde{b}_{\alpha^\vee j} & \text{~~~~~~~~~~~~~~~~~~if $\varepsilon(i)=-1$}; \\ 
\tilde{b}_{\alpha^\vee j}+\Sigma_{k\neq j}[\tilde{b}_{\alpha^\vee k}]_+\tilde{b}_{kj} & \text{~~~~~~~~~~~~~~~~~~if $\varepsilon(i)=1$}.
\end{cases}\\
&=&\begin{cases}
-[\alpha^\vee:\alpha^\vee_j] & \text{if $\varepsilon(i)=-1$}; \\ 
-[\alpha^\vee:\alpha^\vee_j]-\sum_{k\neq j}\tilde{b}_{kj}[\alpha^\vee:\alpha^\vee_k]_+ & \text{if $\varepsilon(i)=1$}.
\end{cases}\\
&=&-\varepsilon(j)[\tau_+(\alpha^\vee):\alpha^\vee_j],
\end{array}
\end{equation}
where the last equality is due to equality (\ref{tau}). As a permutation of $\Phi^\vee_{\geq -1}$, $\tau_+$ induces a permutation of frozen rows of $B^{univ}$. Then the equality (\ref{equ:univ1}) and the equality (\ref{equ:univ}) yield that, up to a sign `$-$', $\mu_-$ preserves the principal part of $B^{univ}$ and induces a permutation of frozen rows of $B^{univ}$, and thus there is an isomorphism $\mu_-(B^{univ})\cong -B^{univ}$ under this permutation.\\

Denote by $\x^{univ}=\{x_j,j\in{\bf{I}};x_{\alpha^\vee},\alpha^\vee\in \Phi^\vee_{\geq -1}\}$ the initial cluster of $\A^{univ}$, define a map $f'_-: \x^{univ}\to \x^{univ}$ by:
\begin{equation*}
f'_-(x_j)=\begin{cases}
\mu_-(x_j) & \text{if $j\in{\bf{I}}$}; \\ 
x_{\tau_+(\alpha^\vee)} & \text{if $\alpha^\vee\in \Phi^\vee_{\geq -1}$}.
\end{cases}
\end{equation*}
Then $f'_-(B^{univ})\cong -{B^{univ}}$. Therefore $f'_-$ induces an inverse cluster automorphism of $\A^{univ}$. Similarly we define $f'_+$, which also induces an inverse cluster automorphism of $\A^{univ}$. Note that $f_-=S'\circ f'_-|_{\x}$, where $S'$ is the specialization from $\A^{univ}$ to $\A$ defined in Definition \ref{def: specialization}. Thus if $f'_-f'_+=1$, then $f_-f_+=1$. Therefore $D_\tau = <f_-,f_+> \subseteq <f'_-,f'_+>\subseteq Aut(\A^{univ})$. If $D_\tau\cong Aut(\A)$, then we have $Aut(\A)\subseteq Aut(\A^{univ})$. Now assume that $D_\tau\subsetneqq Aut(\A)$, by Corollary \ref{cor:simply laced auto group and tau group} and Theorem \ref{thm:tau gp eq to auto gp for nonsimply laced}, $\A$ is of Dynkin type $D_{2n}( n\geqslant3)$ or of Dynkin type $D_4$. For the type $D_{2n} (n\geqslant3)$, exchanging initial cluster variables $x_{-\alpha_{2n-1}}$ and $x_{-\alpha_{2n}}$ induces a direct cluster automorphism  $f$ of $\A$ (see (\ref{eq: f'})), where the rank of $f$ is two. We define a permutation $\sigma$ of $\x^{univ}=\{x_j,j\in{\bf{I}};x_{\alpha^\vee},\alpha^\vee\in \Phi^\vee_{\geq -1}\}$ by:
\begin{equation*}
\sigma(x_j)=\begin{cases}
x_{-\alpha_{2n-1}} & \text{if $j=-\alpha_{2n}$;} \\ 
x_{-\alpha_{2n}} & \text{if $j=-\alpha_{2n-1}$;}\\ 
x_{\alpha^\vee_{2n-1}} & \text{if $j=\alpha^\vee_{2n}$}; \\ 
x_{\alpha^\vee_{2n}} & \text{if $j=\alpha^\vee_{2n-1}$};\\ 
x_j & \text{otherwise}.
\end{cases}
\end{equation*}
Since $\alpha^\vee_{2n-1}$ and $\alpha^\vee_{2n}$ are symmetric in $\Phi^\vee_{\geq -1}$, $B(\sigma(\x^{univ}))\cong-B(\x^{univ})$. Therefore $\sigma$ gives a direct cluster automorphism $f'$ of $\A^{univ}$, where the rank of $f'$ is two. Thus $Aut(\A)\cong D_{\tau}\times <f>\subseteq <f'_-,f'_+>\times <f'>\subseteq Aut(\A^{univ})$. Similarly we can prove the case of type $D_4$. Then $Aut(\A)\subseteq Aut(\A^{univ})$, and finally, we have $Aut(\A)\cong Aut(\A^{univ})$.
\end{proof}

\section*{Acknowledgements}
Both authors thank for anonymous reviewer's careful reading and many valuable suggestions. In particular, Corollary 3.6 comes from a question raised by the reviewer. 

\bigskip\bigskip

{\small Wen Chang\\
School of Mathematics and Information Science, Shaanxi Normal University, Xi'an 710062, China \&\\
Department of Mathematical Sciences, Tsinghua University, Beijing 10084, China\\
Email: {\tt changwen161@163.com}

	\bigskip
Bin Zhu\\
Department of Mathematical Sciences, Tsinghua University, Beijing 10084, China\\
Email: {\tt bzhu@math.tsinghua.edu.cn}}

\end{document}